\documentclass[12pt]{article}

\usepackage{amsmath,amssymb,amsthm}
\usepackage[T1]{fontenc}
\usepackage{newtxtext,newtxmath}
\usepackage[a4paper,margin=1in]{geometry}
\usepackage{setspace}
\usepackage{tikz}
\usetikzlibrary{shapes.geometric}
\usepackage{mathtools}
\usepackage{enumitem}
\usepackage[hidelinks]{hyperref}
\usepackage{microtype}
\usepackage{xcolor}
\usepackage[abbrev,msc-links,nobysame,backrefs]{amsrefs}  
\usetikzlibrary{decorations.pathmorphing}
\usepackage{doi}
\usepackage{float} 

\hypersetup{
    colorlinks=true,
    linkcolor=blue,
    citecolor=blue,
    urlcolor=blue
}
\newtheorem{theorem}{Theorem}[section]
\newtheorem{lemma}[theorem]{Lemma}
\newtheorem{corollary}[theorem]{Corollary}
\newtheorem{proposition}[theorem]{Proposition}
\newtheorem{problem}[theorem]{Problem}
\newtheorem{claim}[theorem]{Claim}
\newtheorem{observation}[theorem]{Observation}

\title{On the maximum number of edges of $k$-cacti}
\author{
 Yuanqiu Huang, Licheng Zhang\thanks{Corresponding author}\\[3pt]
\small School of Mathematics and Statistics, Hunan Normal University\\
\small Changsha, China\\
\small \texttt{hyqq@hunnu.edu.cn}, \texttt{lczhangmath@hunnu.edu.cn}
}

\date{}
\begin{document}
\maketitle

\begin{abstract}
A cactus is a graph in which every edge lies on at most one cycle. In 2024, Zhang and Huang generalized this concept to the $k$-cactus, defined as a graph in which every edge lies on at most $k$ cycles [Discrete Applied Mathematics 348 (2024) 184–191].   It is known that any  cactus on $n$ vertices has  at most $\lfloor\frac{3}{2}(n-1)\rfloor$ edges. However, the upper bound on the size of $k$-cacti was known only for $k\le 4$. In
this note we consider general $k$.  We prove
that every $n$-vertex $k$-cactus has
$O\!\left(\frac{\log k}{\sqrt{\log\log k}}\,n\right)$ edges for all
sufficiently large $k$, and a construction shows this is optimal
up to a factor of $\sqrt{\log\log k}$.
\end{abstract}

\textbf{MSC: 05C35; 05C83 }


\section{Introduction}

Throughout this paper, all graphs are finite, undirected and simple, and are not
required to be connected. 
Let $e$ denote the base of the natural logarithm ($e = 2.718\ldots$).
By $\log k$ we mean the logarithm of $k$ to the base $e$, that is, $\log_e k$.

A \emph{forest} is a graph that contains no cycle; equivalently, it is a
disjoint union of trees, and a \emph{cactus} is a graph in which each
edge lies on at most one cycle. In combinatorial optimization and chemical graph theory, cacti (or directed cacti) are an important class of graphs whose many interesting properties have been extensively studied; see  \cite{MR3599463,MR5017744,MR4281871, MR3969328,MR4208158} for example.
As a natural generalization of both forests and cacti, we call a graph a \emph{$k$-cactus} (for an integer $k\ge 0$) if each of its edges lies on
at most $k$ cycles; this notion was introduced in~\cite{MR4700794}. Thus
a forest is precisely a $0$-cactus and a cactus is a $1$-cactus. Although the notion of a $k$-cactus is relatively new, the underlying idea of
studying a graph through the cycles passing through its edges is an old
one. For example, Toida~\cite{MR0332580} showed
that a graph is Eulerian if and only if every edge lies on an odd number
of cycles. Mac Lane~\cite{MR1546002}
showed that a graph is planar if and only if it has a cycle basis in
which every edge lies in at most two of the basis cycles. More generally, the basis number of a graph is the least
integer $k$ such that the graph has a cycle basis in which every edge
lies in at most $k$ of the basis cycles. In 2026, Geniet and
Giocanti~\cite{GenietGiocanti} showed that the basis number is bounded
by a constant on every proper minor-closed class of graphs.

A forest on $n$ vertices has at most $n-1$ edges, with equality if and
only if it is a tree, and it is well known
that every cactus on $n$ vertices has at most $\lfloor 3(n-1)/2\rfloor$
edges \cite{West2001}. It is natural to seek the analogous upper bound on the number of edges for 
$k$-cacti of a given order. This question was raised in~\cite{MR4700794}, where tight bounds were obtained for $2\le k\le 4$ and the corresponding extremal graphs were characterized. The authors also showed that every $2$-connected $k$-cactus on $n$
vertices was shown to have at most $n+k-1$ edges, with equality when
$n\ge k+2$. The notion of a $k$-cactus has since attracted further attention.
Voblyi~\cite{Voblyi2025Eulerian3Cacti} enumerated labeled Eulerian
$3$-cacti, obtaining exact and asymptotic formulas for their number.
Wu, Chen and Li~\cite{WuChenLi2026} determined the graphs of
maximum spectral radius among all $n$-vertex $2$-cacti and among all  $n$-vertex  $3$-cacti.

However, for general $k$, the extremal problem concerning the maximum number of edges of a $k$-cactus of a given order has not been settled yet.
The purpose of this note is to fill this gap.
As a warm-up, we first observe that a direct argument based on the
block decomposition gives a bound of order $\sqrt{k}\,n$ for the size of
a connected $n$-vertex $k$-cactus (Proposition~\ref{prop:sqrt-bound});
this serves as a benchmark for what follows. To improve the naive upper bound, we start from a key observation: in $K_r$, every edge belongs to exactly $\lfloor e(r-2)!\rfloor-1$ cycles. Hence, a $k$-cactus cannot contain a $K_r$-minor whenever $\lfloor e(r-2)!\rfloor-1 > k$. The local edge-cycle condition thus forces the exclusion of a complete minor whose size depends only on $k$. Finally,  applying the well-known Kostochka--Thomason theorem (Lemma \ref{lem:KT}) gives the bound $\dfrac{\log k}{\sqrt{\log\log k}}\,n$ on the number of edges of a $k$-cactus.



Our main result is as follows.
\begin{theorem}\label{thm:minor-bound}
Let $G$ be an $n$-vertex $k$-cactus, and set
$
        R(k)=\min\{r\ge 3:\ \lfloor e(r-2)!\rfloor-1>k\}.
$
Then $G$ contains no $K_{R(k)}$ minor, and 
$
        |E(G)|\le c\,R(k)\sqrt{\log R(k)}\,n,
$
where $c>0$ is an absolute constant.
\end{theorem}

We can estimate $R(k)$ yielding the following explicit bound. For real-valued functions $f$ and $g$ defined for all sufficiently large
$x$, with $g(x)>0$, we write $f(x)=O(g(x))$ if there are constants $C>0$
and $x_0$ such that $|f(x)|\le C\,g(x)$ for all $x\ge x_0$, and
$f(x)=\Omega(g(x))$ if there are constants $c'>0$ and $x_0$ such that
$f(x)\ge c'\,g(x)$ for all $x\ge x_0$.

\begin{corollary}\label{cor:minor-explicit}
For all sufficiently large $k$, every $n$-vertex $k$-cactus $G$
satisfies
\[
        |E(G)|=O\!\left(\frac{\log k}{\sqrt{\log\log k}}\,n\right).
\]
\end{corollary}

This is close to best possible: coalescing many copies of a complete
graph gives, for infinitely many $n$, $k$-cacti with
$\Omega\!\left(\frac{\log k}{\log\log k}\,n\right)$ edges
(Proposition~\ref{cor:lower-log}). Thus only the power of $\log\log k$
remains to be determined.

We comment that the implicit constant in the $O$-notation can in principle be made explicit. 
For sufficiently large $k$, the proof of Lemma 4.3 gives the bound
$
        R(k) \le \left\lceil 3 \frac{\log k}{\log\log k} \right\rceil + 2.
$
Using this value as an upper bound in Theorem~1.1 and choosing a conservative constant $c=1$ 
yields a safe explicit estimate
$
        |E(G)| \le 2 \, \frac{\log k}{\sqrt{\log\log k}} \, n.
$
We do not attempt to optimize this numerical constant; our focus is on the asymptotic 
dependence of the coefficient of $n$ on $k$.


\subsection{Preliminaries}
In this subsection, we give some definitions, notions and technical lemmas.  For a graph $G$, we write
$V(G)$ and $E(G)$ for its vertex set and edge set, and we call $|E(G)|$
the \emph{size} of $G$. A \emph{cycle} is a connected $2$-regular
subgraph. A \emph{cut-vertex} of a connected graph is a vertex whose deletion
disconnects it. A \emph{block} of a graph is a maximal connected subgraph
without a cut-vertex; thus every block is either a single edge or
a maximal $2$-connected subgraph. The blocks of a graph partition its
edge set.

Let $H$ be a graph. The graph $H$ is a \emph{minor} of a graph $G$ if $H$ can
be obtained from $G$ by a sequence of vertex deletions, edge deletions, and
edge contractions. An \emph{$H$-model} in a graph $G$ consists of a subgraph $M \subseteq G$ and a partition of $V(M)$ into $|V(H)|$ nonempty \emph{branch sets} $X_v$ (one for each vertex $v \in V(H)$), satisfying:
\begin{itemize}
  \item For each $v \in V(H)$, the branch set $X_v$ induces a connected subgraph in $M$;
  \item for each edge $uv \in E(H)$, there is at least one edge in $M$ joining $X_u$ and $X_v$.
\end{itemize}
For each $uv \in E(H)$ we fix one such edge, called the \emph{representative edge} $\widehat{uv}$.


The following lemma gives an equivalent characterization of a graph $G$ containing an $H$-minor.
\begin{lemma}[see {\cite[Section~1.7]{MR4874150}}]\label{lem:model}
A graph $H$ is a minor of $G$ if and only if $G$ contains an $H$-model.
\end{lemma}

\begin{lemma}[Kostochka-Thomason theorem \cite{MR779891}, \cite{MR735367}, \cite{MR1814910}]
\label{lem:KT}
There is an absolute constant $c>0$ such that every $n$-vertex graph with
more than $c\,r\sqrt{\log r}\,n$ edges contains a $K_r$-minor.
\end{lemma}


For any edge $e$ in a graph $G$, let $c_G(e)$ denote the number of cycles in $G$ containing $e$.
\begin{lemma}\label{lem:cycles-edge-Kr}
For $r\ge 3$, the number of cycles of $K_r$ containing a fixed edge is
$
\lfloor e(r-2)!\rfloor-1.
$
\end{lemma}

\begin{proof}
Fix an edge \(uv\) in \(K_r\). There is a bijection between the cycles containing \(uv\) and the \(u\)--\(v\) paths in \(K_r-uv\): adding \(uv\) to a path gives a cycle, and different paths give different cycles. Hence the number of such cycles equals the number of such paths.

Assume that this path has exactly \(j\) internal vertices. These vertices
are chosen from the \(r-2\) vertices in \(V(K_r)\setminus\{u,v\}\), and
then ordered along the path. Thus, for fixed \(j\), the number of such
paths is
$
        \binom{r-2}{j}j!.
$
Therefore
$$
        c_{K_r}(uv)=\sum_{j=1}^{r-2}\binom{r-2}{j}j!.
$$

Put \(m=r-2\). Then
\[
c_{K_r}(uv)=\sum_{j=1}^{m}\binom{m}{j}j!
   =\sum_{j=1}^{m}\frac{m!}{(m-j)!}
   =\sum_{i=0}^{m-1}\frac{m!}{i!}
   =\sum_{i=0}^{m}\frac{m!}{i!}-1 .
\]
We claim that
$
        \lfloor e\,m!\rfloor=\sum_{i=0}^{m}\frac{m!}{i!}.
$
Indeed, by the Taylor expansion
$
        e=\sum_{i=0}^{\infty}\frac1{i!},
$
we have
\[
        e\,m!
        =
        \sum_{i=0}^{m}\frac{m!}{i!}
        +
        \sum_{i=m+1}^{\infty}\frac{m!}{i!}.
\]
The first sum is an integer. For the tail $\sum_{i=m+1}^{\infty}\frac{m!}{i!}$, write \(i=m+s\). Then
\[
        \sum_{i=m+1}^{\infty}\frac{m!}{i!}
        =
        \sum_{s=1}^{\infty}
        \frac{1}{(m+1)(m+2)\cdots(m+s)}.
\]
This tail is positive, and
\[
        \sum_{s=1}^{\infty}
        \frac{1}{(m+1)(m+2)\cdots(m+s)}
        <
        \sum_{s=1}^{\infty}\frac1{(m+1)^s}
        =
        \frac1m
        \le 1.
\]
Thus
$
        \sum_{i=0}^{m}\frac{m!}{i!}
        <
        e\,m!
        <
        \sum_{i=0}^{m}\frac{m!}{i!}+1.
$
Since \(\sum_{i=0}^{m}m!/i!\) is an integer, it follows that
$
        \lfloor e\,m!\rfloor
        =
        \sum_{i=0}^{m}\frac{m!}{i!}.
$
Consequently,
$
       c_{K_r}(uv)
        =
        \lfloor e\,m!\rfloor-1
        =
        \lfloor e\,(r-2)!\rfloor-1.
$
\end{proof}
\begin{observation}\label{ob:1}
$
        \log(m!)\ge m(\log m-1)
$
for every integer \(m\ge 1\).  
\end{observation}

\begin{proof}
It is equivalent to
\(m!\ge (m/e)^m\), which follows from
$
        e^m=\sum_{j=0}^{\infty}\frac{m^j}{j!}\ge \frac{m^m}{m!}.
$
\end{proof}

\begin{lemma}\label{lem:R-upper}
For all sufficiently large \(k\), let
$
        R(k)=\min\left\{r\ge 3:
        \left\lfloor e(r-2)!\right\rfloor-1>k\right\}.
$
Then
$
        R(k)=O\left(\frac{\log k}{\log\log k}\right).
$
\end{lemma}

\begin{proof}

We only need to find one admissible value of \(r\) of order
\(\log k/\log\log k\).  Indeed, since \(R(k)\) is the least integer \(r\)
satisfying
$
        \left\lfloor e(r-2)!\right\rfloor-1>k,
$
any such admissible \(r\) gives an upper bound on \(R(k)\).

Set
$
        m=\left\lceil 3\frac{\log k}{\log\log k}\right\rceil
$
and
  $      r=m+2.
$
We shall prove that this \(r\) is admissible for all sufficiently large \(k\).
By Observation \ref{ob:1}, we have
\begin{equation}\label{eq:1}
 \log(m!)\ge m(\log m-1),
\end{equation}
By the choice of \(m\), we have
\begin{equation}\label{eq:2}
  m\ge 3\frac{\log k}{\log\log k}.
\end{equation}
      Thus
$
        \log m
        \ge
        \log\left(3\frac{\log k}{\log\log k}\right)
        =
        \log 3+\log\log k-\log\log\log k.
$
For all sufficiently large \(k\),
$$
        \log 3+\log\log k-\log\log\log k-1
        \ge
        \frac12\log\log k.
$$
Hence
\begin{equation}\label{eq:3}
        \log m-1\ge \frac12\log\log k.
\end{equation}
Therefore by (\ref{eq:1}), (\ref{eq:2}) and (\ref{eq:3}) one has
\[
        \log(m!)
        \ge
        m(\log m-1)
        \ge
        \left(3\frac{\log k}{\log\log k}\right)
        \left(\frac12\log\log k\right)
        =
        \frac32\log k
        >
        \log k.
\]
Thus \(m!>k\).

Since \(e>2\), we have
$
        \left\lfloor em!\right\rfloor-1
        \ge
        2m!-1
        >
        k.
$
As \(r=m+2\), this means
$
        \left\lfloor e(r-2)!\right\rfloor-1>k.
$
Thus \(r\) is admissible, and so
$
        R(k)\le r=m+2.
$
Finally,
\[
        m+2
        =
        \left\lceil 3\frac{\log k}{\log\log k}\right\rceil+2
        =
        O\left(\frac{\log k}{\log\log k}\right).
\]
Therefore
$
        R(k)=O\left(\frac{\log k}{\log\log k}\right).
$
\end{proof}
\section{A bound from the block decomposition}

First, we use block decomposition to give a simple upper bound on the number of edges in a $k$-cactus on $n$ vertices. The bound serves as a benchmark for our main result.  We first recall the following lemma.

\begin{lemma}[Zhang and Huang \cite{MR4700794}]\label{lem:2conn}
Let $G$ be  a $k$-cactus graph on $n$ vertices. If $G$ is 2-connected, then
$
        |E(G)|\le n+k-1.
$
\end{lemma}

\begin{proposition}\label{prop:sqrt-bound}
Let $G$ be a $n$-vertex $k$-cactus. Then
$
        |E(G)|
        \le
        \frac{3+\sqrt{1+8k}}{4}(n-1).
$
\end{proposition}

\begin{proof} If $G$ is disconnected, we may add bridges between its components (edges joining different components create no new cycles). The resulting graph is still a $k$-cactus, has the same vertex set, and has at least as many edges. Thus it suffices to prove the result for connected
graphs. It is well-known that every connected graph has a block decomposition; that is, its edge set can be partitioned into the edge sets of its blocks. 

Let $\mathcal B$ be the set of blocks of $G$. Clearly, $B$ is a $k$-cactus. For each $B\in\mathcal B$, put $v=|V(B)|$. 

\begin{claim}\label{claim:1}

$
        |E(B)|\le \min\left\{v+k-1,\binom v2\right\}.
$
\end{claim}

\begin{proof}
Since $B$ is $2$-connected, Lemma~\ref{lem:2conn} yields $|E(B)|\le v+k-1$. The bound $\binom{v}{2}$ is the trivial upper bound for any $v$-vertex graph.
\end{proof}

Let
$
    v_0 = \frac{3+\sqrt{1+8k}}{2},
$
which is the larger root of the equation $\binom{x}{2} = x + k - 1$ in $x$.
\begin{claim}\label{claim:2}
$
        |E(B)|\le \frac{v_0}{2}(v-1)
$
for every $B\in\mathcal B$. 
\end{claim}

\begin{proof}
 Put
$
f(x)=\binom{x}{2}$, $ g(x)=x+k-1$ and $
\ell(x)=\frac{v_0}{2}(x-1).
$
By the definition of \(v_0\), we have \(f(v_0)=g(v_0)\). Moreover,
$
\ell(1)=f(1)=0
$
and
$
\ell(v_0)=\frac{v_0}{2}(v_0-1)=\binom{v_0}{2}=f(v_0)=g(v_0).
$
Thus \(\ell\) is the secant line of the convex function \(f\) between
\(x=1\) and \(x=v_0\). Hence
$
f(x)\le \ell(x)$ for $1\le x\le v_0.
$
On the other hand, the two linear functions \(g\) and \(\ell\) agree at
\(x=v_0\), and their slopes are \(1\) and \(v_0/2\), respectively. Since
\(v_0\ge 2\), we have \(v_0/2\ge 1\). Therefore
$
g(x)\le \ell(x)$ for $x\ge v_0.
$
Consequently, for every \(v\ge 1\),
$
\min\left\{g(v),f(v)\right\}\le \ell(v)
=\frac{v_0}{2}(v-1).
$
By Claim \ref{claim:1}, this gives
$
|E(B)|\le \frac{v_0}{2}(v-1).
$
\end{proof}

%
%
%
%
%
%
%
Finally, the block decomposition gives
$
        |E(G)|=\sum_{B\in\mathcal B}|E(B)|,
$ and $
        n-1=\sum_{B\in\mathcal B}\bigl(|V(B)|-1\bigr).
$
Hence, we have
\[
        |E(G)|
        \le  
        \sum_{B\in\mathcal B}   \frac{v_0}{2}\bigl(|V(B)|-1\bigr)
        =
        \frac{v_0}{2}(n-1).
\]
Since
$
        \frac{v_0}{2}
        =
        \frac{3+\sqrt{1+8k}}{4},
$
we obtain
$
        |E(G)|
        \le
        \frac{3+\sqrt{1+8k}}{4}(n-1).
$
\end{proof}

While Proposition~\ref{prop:sqrt-bound} gives a bound of roughly $\sqrt{k}\, n$, we can do much better for large $k$: the minor-based method in next Section~\ref{sec:3} replaces this with a coefficient of order $\frac{\log k}{\sqrt{\log\log k}}$.

\section{Proofs of Theorem \ref{thm:minor-bound} and Corollary \ref{cor:minor-explicit}}\label{sec:3}

Before proving our main theorems, we need a preparatory subsection. Proposition~\ref{lem:lift} will be a crucial tool; and we also show that it implies that cactus graphs are minor-closed (see Theorem \ref{cor:closed}) which is of independent interest.

\subsection{ Minor-closedness of $k$-cacti}

A class of graphs is \emph{minor-closed} if every minor of a graph in the
class again belongs to the class.
It is known that forests and cacti are minor-closed: a forest is
precisely a graph with no $K_3$-minor, and a cactus  is
precisely a graph with no diamond-minor~\cite{ElMallahColbourn1988}, where
the diamond is the graph obtained from $K_4$ by deleting an edge. First, we shall show
that the class of $k$-cacti is minor-closed for every
fixed $k$. The following proposition will be a key tool.
\begin{proposition}\label{lem:lift}
Let $H$ be a minor of a graph $G$. Then for every edge $h\in E(H)$ there exists an edge
$\widehat h\in E(G)$ such that
$
        c_H(h)\le c_G(\widehat h).
$
\end{proposition}
\begin{proof}
By Lemma \ref{lem:model}, we can
fix an $H$-model in $G$ with branch sets $\{X_x\}_{x\in V(H)}$. For each
edge $xy\in E(H)$, fix one representative  edge of $xy$ of $G$ joining $X_x$ and $X_y$,  denoted $\widehat{xy}$. Since the branch
sets are disjoint, distinct edges of $H$ have distinct representative
edges.

Fix an edge $h\in E(H)$. If $h$ lies on no cycle of $H$, then $c_H(h)=0$ and the inequality
holds trivially for $\widehat h$. Assume therefore that $h$ belongs to at least one cycle
and pick a cycle $C=x_1x_2\cdots x_tx_1$ of $H$ through $h$ with $h=x_1x_2$ (say). Replace each edge $x_ix_{i+1}$ by its representative edge
$\widehat{x_ix_{i+1}}$, and inside each branch set \(X_{x_i}\), join the two endvertices incident with
the representative edges by a path, possibly of length zero if the endvertices
coincide; this is possible since \(X_{x_i}\) is connected.
These representative edges and paths form a cycle of $G$ through
$\widehat h$ (see Figure~\ref{fig:lift}): the paths lie in pairwise
disjoint branch sets and the representative edges run between consecutive
branch sets, so no vertex repeats. Distinct cycles of $H$ through $h$
give distinct cycles of $G$, since they differ in some edge $e$, hence in
its representative edge $\widehat e$.
Therefore the cycles of $H$ through $h$ inject into the cycles of $G$
through $\widehat h$, and $c_H(h)\le c_G(\widehat h)$.
\end{proof}

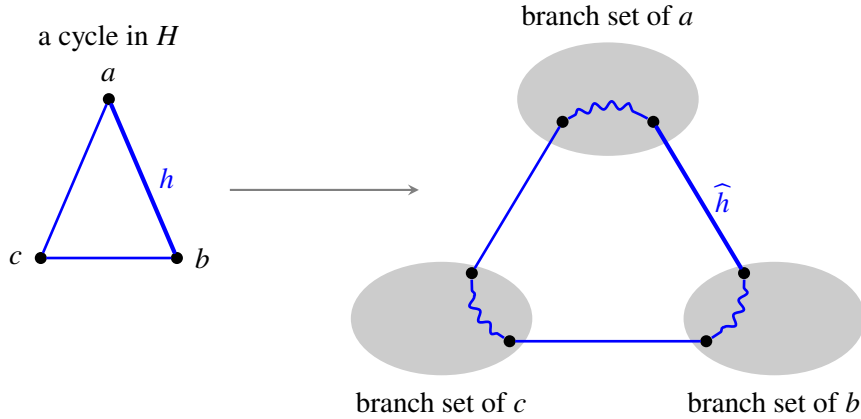
\begin{figure}[H]
\centering
\begin{tikzpicture}[
  every node/.style={font=\small},
  bs/.style={fill=lightgray!80, dashed, ellipse, inner sep=0pt},
  dot/.style={circle, fill, inner sep=1.6pt},
  cyc/.style={blue, line width=1pt, line cap=round},
  hed/.style={blue, line width=1.6pt, line cap=round},
  pth/.style={blue, line width=1pt, decorate,
    decoration={snake, amplitude=0.5mm, segment length=2.2mm,
      pre length=1mm, post length=1mm}}]
\node[dot] (a) at (0,1.5) {};
\node[dot] (b) at (0.9,-0.6) {};
\node[dot] (c) at (-0.9,-0.6) {};
\node[above] at (a.north) {$a$};
\node[right] at (b.east) {$b$};
\node[left]  at (c.west) {$c$};
\draw[hed] (a) -- (b) node[midway, right=1pt] {$h$};
\draw[cyc] (b) -- (c);
\draw[cyc] (c) -- (a);
\node at (0,2.3) {a cycle in $H$};
\draw[->, >=stealth, thick, gray] (1.6,0.3) -- (4.1,0.3);
\begin{scope}[xshift=1.2cm]
\node[bs, minimum width=2.4cm, minimum height=1.5cm] (Ba) at (5.4,1.5) {};
\node[bs, minimum width=2.4cm, minimum height=1.5cm] (Bb) at (7.6,-1.4) {};
\node[bs, minimum width=2.4cm, minimum height=1.5cm] (Bc) at (3.2,-1.4) {};
\node[above=2pt] at (Ba.north) {branch set of $a$};
\node[below=2pt] at (Bb.south) {branch set of $b$};
\node[below=2pt] at (Bc.south) {branch set of $c$};
\node[dot] (a1) at (4.8,1.2) {};
\node[dot] (a2) at (6.0,1.2) {};
\node[dot] (b1) at (7.2,-0.8) {};
\node[dot] (b2) at (6.7,-1.7) {};
\node[dot] (c1) at (4.1,-1.7) {};
\node[dot] (c2) at (3.6,-0.8) {};
\draw[hed] (a2) -- (b1) node[midway, right=1pt] {$\widehat h$};
\draw[cyc] (b2) -- (c1);
\draw[cyc] (c2) -- (a1);
\draw[pth] (a1) to[bend left=30] (a2);
\draw[pth] (b1) to[bend left=30] (b2);
\draw[pth] (c1) to[bend left=30] (c2);
\end{scope}
\end{tikzpicture}
\caption{Lifting a cycle of $H$ to a cycle of $G$ through a minor model.}
\label{fig:lift}
\end{figure}

\begin{theorem}\label{cor:closed}
 If a graph $G$ is a $k$-cactus and $H$ is a
minor of $G$, then $H$ is a $k$-cactus.
\end{theorem}

\begin{proof}
Let $h\in E(H)$ be arbitrary. By Proposition~\ref{lem:lift} there is $\widehat h\in
E(G)$ with $c_H(h)\le c_G(\widehat h)$, and $c_G(\widehat h)\le k$ because $G$ is
a $k$-cactus. Hence $c_H(h)\le k$ for every edge $h$, i.e.\ $H$ is a $k$-cactus.
\end{proof}

\subsection{Proofs of Theorem \ref{thm:minor-bound} and Corollary \ref{cor:minor-explicit}}

\begin{lemma}\label{lem:minor-cycle}
If a graph \(G\) contains \(K_r\) as a minor, then some edge of \(G\) is contained in at least
$
        \left\lfloor e(r-2)!\right\rfloor-1
$
cycles.
\end{lemma}

\begin{proof}
By Lemma~\ref{lem:cycles-edge-Kr}, every edge of \(K_r\) lies in exactly
\(\lfloor e(r-2)!\rfloor-1\) cycles.  Fix one such edge \(h\) and apply the
 Proposition~\ref{lem:lift} with \(H=K_r\): there is an edge
\(\widehat h\in E(G)\) with
$
        c_G(\widehat h)\ge c_{K_r}(h)=\left\lfloor e(r-2)!\right\rfloor-1. 
$
\end{proof}

\begin{lemma}\label{lem:no-minor}
Let \(G\) be a \(k\)-cactus.  If
$
        \left\lfloor e(r-2)!\right\rfloor-1>k,
$
then \(G\) contains no \(K_r\)-minor.
\end{lemma}

\begin{proof}
If \(G\) contained a \(K_r\)-minor, then by Lemma~\ref{lem:minor-cycle}, some edge of \(G\) would be contained in at least
$
        \left\lfloor e(r-2)!\right\rfloor-1
$
cycles.  Since this number is larger than \(k\), this contradicts the definition of a \(k\)-cactus.  Hence \(G\) contains no \(K_r\)-minor.
\end{proof}

\begin{proof}[\textbf{Proof of Theorem~\ref{thm:minor-bound}}]
By the definition of \(R(k)\),
$
        \left\lfloor e(R(k)-2)!\right\rfloor-1>k,
$
so by Lemma~\ref{lem:no-minor}, \(G\) contains no \(K_{R(k)}\)-minor. If
$
        |E(G)|>c\,R(k)\sqrt{\log R(k)}\,n,
$
then Lemma~\ref{lem:KT} would give a \(K_{R(k)}\)-minor in \(G\), a
contradiction. Hence
$
        |E(G)|\le c\,R(k)\sqrt{\log R(k)}\,n. 
$
\end{proof}


\begin{proof}[\textbf{Proof of Corollary~\ref{cor:minor-explicit}}]
Let $R=R(k)$. By Theorem~\ref{thm:minor-bound},
$
        |E(G)|\le c\,R\sqrt{\log R}\,n.
$
For all sufficiently large \(k\), Lemma~\ref{lem:R-upper} gives
$
        R=O\!\left(\frac{\log k}{\log\log k}\right),
$
and in particular
$
        \log R=O(\log\log k).
$
Hence
\[
\begin{aligned}
        R\sqrt{\log R}
        &=
        O\!\left(
        \frac{\log k}{\log\log k}\sqrt{\log\log k}
        \right) \\
        &=
        O\!\left(
        \frac{\log k}{\sqrt{\log\log k}}
        \right).
\end{aligned}
\]
Since $c$ is an absolute constant, combining this with
Theorem~\ref{thm:minor-bound} gives
$
        |E(G)|
        =
        O\!\left(
        \frac{\log k}{\sqrt{\log\log k}}\,n
        \right)
$
for all sufficiently large \(k\).
\end{proof}

\section{Concluding remarks}

For an $n$-vertex $k$-cactus $G$, we have shown that
$
        |E(G)|
        =
        O\!\left(\frac{\log k}{\sqrt{\log\log k}}\,n\right).
$ We note, however, that this bound is unlikely to be optimal, and further improvement seems possible.
There are infinitely many $n$ for which Proposition~\ref{cor:lower-log} gives $k$-cacti with $|E(G)| = \Omega\!\left(\frac{\log k}{\log\log k}\,n\right)$.
To prove Proposition~\ref{cor:lower-log}, we need the following lemma.
\begin{lemma}\label{prop:lower}
Let \(r\ge3\) be an integer satisfying
$
        \left\lfloor e(r-2)!\right\rfloor-1\le k.
$
Then, for infinitely many values of \(n\), there exists an \(n\)-vertex \(k\)-cactus \(G\) with
$
        |E(G)|=\frac r2(n-1).
$
\end{lemma}

\begin{proof}
By Lemma~\ref{lem:cycles-edge-Kr}, every edge of \(K_r\) is contained in exactly
$
        \left\lfloor e(r-2)!\right\rfloor-1
$
cycles.  Hence, under the assumed inequality, \(K_r\) is a \(k\)-cactus.
Take \(s\) copies of \(K_r\), and identify one vertex from each copy into a single common vertex.  The resulting graph is still a \(k\)-cactus, since every cycle lies entirely inside one of the complete-graph blocks.
The graph has
$
        n=1+s(r-1)
$
vertices and
$
        m=s\binom r2
$
edges.  Therefore
$
        m
        =
        s\frac{r(r-1)}2
        =
        \frac r2\,s(r-1)
        =
        \frac r2(n-1).
$
As \(s\) ranges over the positive integers, this gives infinitely many values of \(n\).
\end{proof}

\begin{proposition}\label{cor:lower-log}
For every sufficiently large $k$, there are infinitely many values of $n$
for which there exists an $n$-vertex $k$-cactus $G$ satisfying
$
        |E(G)|
        =
        \Omega\left(
        \frac{\log k}{\log\log k}\, n
        \right).
$
\end{proposition}

\begin{proof}

Choose
$
        m=\left\lfloor
        \frac12\frac{\log k}{\log\log k}
        \right\rfloor,
$ and $
        r=m+2.
$
For all sufficiently large $k$, we have $m\ge 1$, and hence $r\ge 3$, and now we  prove that   $ \left\lfloor e(r-2)!\right\rfloor-1\le k$.  Since $r-2=m$, it is
enough to show that
$
        \lfloor em!\rfloor-1\le k.
$
By the definition of $m$,
$
        m\le \frac12\frac{\log k}{\log\log k}.
$
Also, for all sufficiently large $k$, we have
$
        m\le \log k.
$
Therefore
$
        \log m\le \log\log k.
$
Since $m!\le m^m$, it follows that
$
        \log(m!)\le m\log m.
$
Combining the above inequalities gives
\[
        \log(m!)
        \le
        m\log m
        \le
        \left(\frac12\frac{\log k}{\log\log k}\right)\log\log k
        =
        \frac12\log k.
\]
Hence
$
        m!\le e^{(\log k)/2}=\sqrt{k}.
$
Thus
$
        em!\le e\sqrt{k}.
$
For all sufficiently large $k$, we have $e\sqrt{k}\le k$.  Consequently,
$
        em!\le k,
$
and hence
$
        \lfloor em!\rfloor-1\le k.
$
Therefore
$
        \left\lfloor e(r-2)!\right\rfloor-1\le k,
$ as desired.

 Thus, by Lemma~\ref{prop:lower}, for infinitely many values
of $n$, there exists an $n$-vertex $k$-cactus $G$ such that
$
        |E(G)|=\frac r2(n-1).
$
Now we prove that $
        |E(G)|
        =
        \Omega\left(
        \frac{\log k}{\log\log k}\,n
        \right).
$  Since
$
        \frac12\frac{\log k}{\log\log k}\to +\infty
$ as $ k\to +\infty,
$
we have, for all sufficiently large $k$,
$
        \left\lfloor
        \frac12\frac{\log k}{\log\log k}
        \right\rfloor
        \ge
        \frac14\frac{\log k}{\log\log k}.
$
Thus
$
        m\ge \frac14\frac{\log k}{\log\log k}.
$
Since $r=m+2\ge m$, we obtain 
$
        r\ge m\ge\frac14\frac{\log k}{\log\log k}.
$
We may assume $n\ge 2$, and so
$
        n-1\ge \frac n2.
$
Therefore
$
        |E(G)|
        =
        \frac r2(n-1)
        \ge
        \frac r2\cdot \frac n2
        =
        \frac r4 n.
$
Using the lower bound on $r$, we get
$
        |E(G)|
        \ge
        \frac1{16}\frac{\log k}{\log\log k}\,n.
$
Hence
$
        |E(G)|
        =
        \Omega\left(
        \frac{\log k}{\log\log k}\,n
        \right).
$
This proves the proposition.
\end{proof}

By Proposition~\ref{cor:lower-log}, we even suspect that our upper bound can be further improved.

\begin{problem}
Determine the asymptotic order of the maximum number of edges in an
$n$-vertex $k$-cactus. In particular, does every $n$-vertex $k$-cactus
$G$ satisfy
$
        |E(G)|
        =
        O\!\left(\frac{\log k}{\log\log k}\,n\right)?
$
\end{problem}

\section*{Acknowledgment}

We claim that there is no conflict of interest in our paper. No data was used for the research
described in the article.
The work was supported by the National Natural Science Foundation of China
(Grant Nos. 12271157 and 12371346).


\begin{bibdiv}
\begin{biblist}
\bib{MR4874150}{book}{
   author={Diestel, Reinhard},
   title={Graph theory},
   series={Graduate Texts in Mathematics},
   volume={173},
   edition={6},
   publisher={Springer, Berlin},
   date={[2025] \copyright 2025},
   pages={xx+454},
   isbn={978-3-662-70106-5},
   isbn={978-3-662-70107-2},
   review={\MR{4874150}},
}
\bib{ElMallahColbourn1988}{article}{
   author={El-Mallah, Ehab},
   author={Colbourn, Charles J.},
   title={The complexity of some edge deletion problems},
   journal={IEEE Transactions on Circuits and Systems},
   volume={35},
   date={1988},
   number={3},
   pages={354--362},
   doi={10.1109/31.1748},
}
\bib{GenietGiocanti}{article}{
   author={Geniet, Colin},
   author={Giocanti, Ugo},
   title={Basis number of graphs excluding minors},
   date={2026},
   eprint={2601.05195},
   archive={arXiv},
   doi={10.48550/arXiv.2601.05195},
}

\bib{MR3599463}{article}{
   author={He, Fangguo},
   author={Zhu, Zhongxun},
   title={Cacti with maximum eccentricity resistance-distance sum},
   journal={Discrete Appl. Math.},
   volume={219},
   date={2017},
   pages={117--125},
   issn={0166-218X},
   review={\MR{3599463}},
   doi={10.1016/j.dam.2016.10.032},
}
\bib{MR5017744}{article}{
   author={Hu, Ran},
   author={Kanani, Divy H.},
   author={Zhang, Jingru},
   title={Computing the center of uncertain points on cactus graphs},
   journal={Theoret. Comput. Sci.},
   volume={1067},
   date={2026},
   pages={Paper No. 115761, 11},
   issn={0304-3975},
   review={\MR{5017744}},
   doi={10.1016/j.tcs.2026.115761},
}

\bib{MR4281871}{article}{
   author={Jiang, Yisheng},
   author={Lu, Mei},
   title={A note on the minimum inverse sum indeg index of cacti},
   journal={Discrete Appl. Math.},
   volume={302},
   date={2021},
   pages={123--128},
   issn={0166-218X},
   review={\MR{4281871}},
   doi={10.1016/j.dam.2021.06.011},
}
\bib{MR779891}{article}{
   author={Kostochka, A. V.},
   title={Lower bound of the Hadwiger number of graphs by their average degree},
   journal={Combinatorica},
   volume={4},
   date={1984},
   number={4},
   pages={307--316},
   issn={0209-9683},
   review={\MR{779891}},
   doi={10.1007/BF02579141},
}
\bib{MR3969328}{article}{
   author={Li, Shuchao},
   author={Zhang, Licheng},
   author={Zhang, Minjie},
   title={On the extremal cacti of given parameters with respect to the
   difference of zagreb indices},
   journal={J. Comb. Optim.},
   volume={38},
   date={2019},
   number={2},
   pages={421--442},
   issn={1382-6905},
   review={\MR{3969328}},
   doi={10.1007/s10878-019-00391-4},
}

\bib{MR1546002}{article}{
   author={Mac Lane, Saunders},
   title={A structural characterization of planar combinatorial graphs},
   journal={Duke Math. J.},
   volume={3},
   date={1937},
   number={3},
   pages={460--472},
   issn={0012-7094},
   review={\MR{1546002}},
   doi={10.1215/S0012-7094-37-00336-3},
}



\bib{MR735367}{article}{
   author={Thomason, Andrew},
   title={An extremal function for contractions of graphs},
   journal={Math. Proc. Cambridge Philos. Soc.},
   volume={95},
   date={1984},
   number={2},
   pages={261--265},
   issn={0305-0041},
   review={\MR{735367}},
   doi={10.1017/S0305004100061521},
}


\bib{MR1814910}{article}{
   author={Thomason, Andrew},
   title={The extremal function for complete minors},
   journal={J. Combin. Theory Ser. B},
   volume={81},
   date={2001},
   number={2},
   pages={318--338},
   issn={0095-8956},
   review={\MR{1814910}},
   doi={10.1006/jctb.2000.2013},
}
\bib{MR0332580}{article}{
   author={Toida, S.},
   title={Properties of a Euler graph},
   journal={J. Franklin Inst.},
   volume={295},
   date={1973},
   pages={343--345},
   issn={0016-0032},
   review={\MR{0332580}},
   doi={10.1016/0016-0032(73)90046-X},
}

\bib{Voblyi2025Eulerian3Cacti}{article}{
   author={Voblyi, V. A.},
   title={Enumeration of labeled Eulerian \(3\)-cacti},
   journal={J. Math. Sci.},
   volume={293},
   date={2025},
   pages={673--677},
   issn={1072-3374},
   doi={10.1007/s10958-025-08041-3},
}


\bib{West2001}{book}{
   author={West, Douglas B.},
   title={Introduction to Graph Theory},
   edition={2},
   publisher={Prentice Hall},
   address={Upper Saddle River, NJ},
   date={2001},
   isbn={0-13-014400-2},
   review={\MR{1367739}},
}
\bib{WuChenLi2026}{article}{
   author={Wu, Zeyuan},
   author={Chen, Hongzhang},
   author={Li, Jianxi},
   title={Maximizing the spectral radius of generalized cactus graphs},
   journal={RAIRO Oper. Res.},
   volume={60},
   date={2026},
   number={3},
   pages={1407--1418},
   issn={0399-0559},
   doi={10.1051/ro/2026035},
}

\bib{MR4208158}{article}{
   author={Zaman, Shahid},
   title={Cacti with maximal general sum-connectivity index},
   journal={J. Appl. Math. Comput.},
   volume={65},
   date={2021},
   number={1-2},
   pages={147--160},
   issn={1598-5865},
   review={\MR{4208158}},
   doi={10.1007/s12190-020-01385-w},
}




\bib{MR4700794}{article}{
   author={Zhang, Licheng},
   author={Huang, Yuanqiu},
   title={On the sizes of generalized cactus graphs},
   journal={Discrete Appl. Math.},
   volume={348},
   date={2024},
   pages={184--191},
   issn={0166-218X},
   review={\MR{4700794}},
   doi={10.1016/j.dam.2024.01.043},
}


\end{biblist}
\end{bibdiv}
\end{document}